\documentclass[twoside, 12pt]{article}
\usepackage[russian, english]{babel}
\usepackage{epsfig}
\usepackage{amssymb,amsmath,amsfonts,soul,amsthm,enumerate}
\usepackage{amsfonts}
\usepackage{amsmath}
\usepackage{graphicx}
\usepackage{amsthm}
\usepackage[cp1251]{inputenc}
\usepackage[T2A]{fontenc}
\usepackage{mathrsfs}
\usepackage[english]{babel}
\newtheorem{theorem}{Theorem}[section]
\newtheorem{lemma}{Lemma}[section]

\numberwithin{equation}{section}
 \def\@evenhead{\vbox{\hbox to \textwidth{\thepage\hfil\sl\leftmark\strut}\hrule}}
 \def\@oddhead{\vbox{\hbox to \textwidth{\rightmark\hfill\thepage\strut}\hrule}}

\usepackage{textcomp}
\usepackage{upgreek}

\theoremstyle{plain} 
\newtheorem{proposition}{Proposition}

\newtheorem{PJf}{Classical Poisson\,--\,Jensen formula}

\theoremstyle{definition}

\usepackage{latexsym,amssymb,amsmath}
\usepackage{amsfonts,mathrsfs,amscd}
\usepackage{verbatim}
\usepackage[mathscr]{eucal}

\newcommand{\RR}{\mathbb{R}} 
 
\newcommand{\NN}{\mathbb{N}}

\newcommand{\pt}{{\rm pt}} 
\newcommand{\dd}{\,{\rm d}}

\DeclareMathOperator{\clos}{clos} 
\DeclareMathOperator{\Int}{int}
\DeclareMathOperator{\Meas}{Meas}

\DeclareMathOperator{\har}{har}

\DeclareMathOperator{\comp}{cmp}
 
\DeclareMathOperator{\sbh}{sbh}

\DeclareMathOperator{\supp}{supp}

\DeclareMathOperator{\Borel}{Bor}

\DeclareMathOperator{\Dom}{Dom}

\DeclareMathOperator{\Conn}{Conn}

\begin{document}
\sloppy

\centerline{\bf UNIQUENESS THEOREMS FOR SUBHARMONIC FUNCTIONS}     

\vskip 0.3cm

\centerline{\bf B. N. Khabibullin}        

\markboth{\hfill{\footnotesize\rm   B. N. Khabibullin}\hfill}
{\hfill{\footnotesize\sl  Uniqueness theorems for subharmonic functions}\hfill}
\vskip 0.3cm

\vskip 0.7 cm

\vskip 0.2cm


\vskip 0.2cm

\noindent {\bf Abstract.} It is shown that harmonic functions on some subsets, subharmonic and coinciding everywhere outside of these sets, actually coincide everywhere.  
Our main result is Theorem \ref{lemPQ}.

\section{Introduction}\label{Int}

\subsection{ On the classical Poisson-Jensen formula}

Let $D$\/ be a {\it bounded domain\/} in the $d$-dimensional Euclidean space  $\RR^d$ with the {\it closure\/} $\clos D$ in $\RR^d$ and the {\it boundary\/} $\partial D$ in $\RR^d$. Then, for any $x\in D$  there are the {\it extended harmonic measure $\omega_D(x, \cdot)$ for  $D$ at $x\in D$\/} 
as a Borel probability measure on $\RR^d$ with support on
$\partial D$ and the {\it generalized Green's function $g_D(\cdot , x )$  for $D$ with pole at $x$\/}  extended by zero values on  the complement $\RR^d\setminus  \clos D$ and by the upper semicontinuous regularization on $\partial D$ from $D$ \cite{HK}, \cite{AG}, \cite{R}, \cite{Helms}, \cite{Doob}, \cite{Landkoff} (see also \eqref{oB} and \eqref{gD} in Subsec.~\ref{ccPJ} below). 

Let $u\not\equiv -\infty$ be  a {\it subharmonic function on\/}  $\clos D$, i.e., on an open set containing $\clos D$,  with its {\it Riesz measure\/} $\varDelta_u$ on this open set (see in detail \S\S~\ref{Sssm}--\ref{Sssf}, \eqref{df:cm}).

\begin{PJf}[{\rm \cite[Theorem 5.27]{HK}, \cite[4.5]{R}}] 
\begin{equation}\label{clasPJ}
u(x)=\int_{\partial D} u\dd \omega_D(x,\cdot)-\int_{\clos D} g_D(\cdot,x)\dd \varDelta_u \quad\text{for each $x\in D$.} 
\end{equation}
\end{PJf}

 For $s\in \RR$, we set  
\begin{subequations}\label{kK}
\begin{align}
k_s(t)& := \begin{cases}
\ln t  &\text{ if $s=0$},\\
-\frac{s}{|s|}
t^{-s} &\text{ if $s\in \RR\setminus 0$,} 
\end{cases}
\qquad  t\in \RR^+\setminus 0,
\tag{\ref{kK}k}\label{{kK}k}
\\
K_{d-2}(y,x)&:=\begin{cases}
k_{d-2}\bigl(|y-x|\bigr)  &\text{ if $y\neq x$},\\
 -\infty &\text{ if $y=x$ and $d\geq 2$},\\
0 &\text{ if $y=x$ and  $d=1$},\\
\end{cases}
\quad  (y,x) \in \RR^d\times \RR^d.
\tag{\ref{kK}K}\label{{kK}K}
\end{align}
\end{subequations}
The following functions 
\begin{equation}\label{pqKo}
p\colon y\underset{\text{\tiny $y\in \RR^d$}}{\longmapsto}  g_D(y,x)+K_{d-2}(y,x), 
\quad   q\colon y\underset{\text{\tiny $y\in \RR^d$}}{\longmapsto} K_{d-2}(y,x)  
\end{equation}
are subharmonic with Riesz probability measures 
$\varDelta_p=\omega_D(x, \cdot)$ and $\varDelta_q=\delta_x$,
where $\delta_x$ is the  Dirac measure at $x\in D$: $\delta_x\bigl(\{x\}\bigr)=1$.
The following symmetric equivalent form of the classical Poisson\,--\,Jensen formula \eqref{clasPJ}  immediately follows  from the suitable definitions of  harmonic measures and Green's functions and is briefly discussed in  Subsec.~\ref{ccPJ}.

\section{ Basic notation, definitions, and conventions}\label{Ss11}

 The reader can skip this Subsec. \ref{Ss11}
and return to it only if necessary.

\subsection{Sets, topology, order}
We denote by $\NN:=\{1,2,\dots\}$, $\RR$, and $\RR^+:=\{x\in \RR\colon x\geq 0\}$  the sets of {\it natural,\/} of {\it real,\/} and  of {\it positive\/} numbers, each endowed with its natural order ($\leq$, $\sup/\inf$), algebraic, geometric  and topological structure.  We denote singleton sets by a symbol without curly brackets. So, $\NN_0:=\{0\}\cup \NN=:0\cup \NN$, and  $\RR^+\setminus 0:=\RR^+\setminus \{0\}$ is the set of {\it strictly positive\/} numbers, etc. The {\it extended real line\/}  $\overline \RR:=-\infty\sqcup\RR\sqcup+\infty$ is the order completion of $\RR$ by the  {\it disjoint union\/} $\sqcup$  with $+\infty:=\sup \RR$ and $-\infty:=\inf \RR$ equipped with the order topology with two  ends $\pm\infty$, $\overline \RR^+:=\RR^+\sqcup+\infty$;  $\inf \varnothing :=+\infty$, $\sup \varnothing :=-\infty$ for the {\it empty set\/} $\varnothing$ etc. 
The same symbol $0$ is also used, depending on the context, to denote  zero vector, zero function, zero measure, etc.

We denote by $\RR^d$ the  {\it Euclidean space of $d\in \NN$ dimensions\/}  with the  {\it Euclidean norm\/} $|x|:=\sqrt{x_1^2+\dots+x_d^2}$ of $x=(x_1,\dots ,x_d)\in \RR^d$,  and  we denote  by $\RR^d_{\infty}:=\RR^d\sqcup\infty$
 the {\it  Alexandroff\/} (Aleksandrov) 
{\it one-point compactification\/}   of $\RR^d$
obtained by adding one extra point $\infty$. For a subset $S\subset \RR^d_{\infty}$ or a subset $S\subset \RR^d$ we let $\complement S :=\RR^d_{\infty}\setminus S$, $\clos S$, $\Int S:=\complement (\clos \complement  S)$, and $\partial S:=\clos S\setminus \Int S$ denote its
 {\it complement,\/} {\it closure,} {\it interior,} and {\it boundary\/}  always in $\RR^d_{\infty}$, and $S$ is equipped with the topology induced from $\RR^d_{\infty}$. If $S'$ is a relative compact subset in $S$, i.e., $\clos S'\subset S$,  then we write $S'\Subset S$.  We denote by 
 $B(x,t):=\{y\in \RR^d\colon |y-x|< t\}$, $\overline B(x,t):=\{y\in \RR^d\colon |y-x|\leq  t\}$, $\partial \overline B(x,t):=\overline B(x,t)\setminus   B(x,t)$  an {\it open ball,\/} a {\it closed ball,\/} 
a {\it sphere 
of radius $t\in \RR^+$ centered at $x\in \RR^d$}, respectively. 

Let $T$ be a topological space, and $S$ be a subset in $T$.   
We denote by $\Conn_T S$ or $\Conn_T (S)$ the  set of all connected components of $S\subset T$ in $T$. 

\underline{Throughout this paper} $O\neq \varnothing$ will denote  an  {\it open subset  in\/ $\RR^d$},   and $D\neq \varnothing$ is a  {\it domain in\/ $\RR^d$,\/}  i.e., an open connected subset in $\RR^d$. 

\subsection{Measures and charges}\label{Sssm}

The convex cone over $\RR^+$ of all Borel, or Radon,  positive measures $\mu\geq 0$  on the $\sigma$-algebra $\Borel (S)$ of all {\it Borel subsets\/} of $S$ is denoted by $\Meas^+(S)$; $\Meas^+_{\comp}(S)\subset \Meas^+(S)$ is the subcone of $\mu\in \Meas^+(S)$ with compact  {\it support\/} $\supp \mu$ in $S$, $\Meas(S):=\Meas^+(S)-\Meas^+(S)$ is the vector lattice over $\RR$ of {\it charges,\/} or signed measures, on $S$, $\Meas^{+1}(S)
$ is the convex set of {\it probability\/} measures on $S$, 
$\Meas_{\comp}^{1+}(S):=\Meas^{1+}(S)\cap \Meas_{\comp}(S)$,
and   $\Meas_{\comp}(S):=\Meas^+_{\comp}(S)-\Meas^+_{\comp}(S)$.
 For a charge $\mu \in \Meas(S)$, we let
$\mu^+:=\sup\{0,\mu\}$, $\mu^-:=(-\mu)^+$ and $\mu:= \mu^++\mu^-$ respectively denote its {\it upper,} {\it lower,} and {\it total variations,} and $\mu(x,t):=\mu\bigl( \overline B(x,t)\bigr)$. 

For an {\it extended numerical  function\/} $f\colon S\to \overline \RR$ we allow values $\pm\infty$ for Lebesgue integrals \cite[Ch.~3, Definiftion 3.3.2]{HK} (see also \cite{Bourbaki})
\begin{equation}\label{int}
\int_Sf\dd \mu\in \overline \RR, \quad \mu \in \Meas^+(S),
\end{equation}
and we say that $f$ is {\it $\mu$-integrable on\/} $S$ if  the integral in \eqref{int} is finite.

\subsection{Subharmonic functions}\label{Sssf}

We denote  by $\sbh (O)$  the convex cone over $\RR^+$  of all   {\it subharmonic\/} (locally convex if $d = 1$) functions on $O$, including functions that are identically equal to $-\infty$ on some components $C\in \Conn_{\RR^d_{\infty}}(O)$. Thus, $\har(O):=\sbh(O)\cap \bigl(-\sbh(O)\bigr)$
is the vector space over $\RR$ of all {\it harmonic\/} (locally affine if $d = 1$) functions on $O$.
Each function \begin{equation*}
u\in \sbh_*(  O):=\bigl\{u\in \sbh( O)\colon u\not\equiv-\infty 
\text{ on each }C\in \Conn_{\RR^d_{\infty}}(O)\bigr\}
\end{equation*}
is associated with its {\it Riesz measure\/}
\begin{equation}\label{df:cm}
\varDelta_u:= c_d {\bigtriangleup}  u\in \Meas^+( O), 
\quad c_d:=\frac{\Gamma(d/2)}{2\pi^{d/2}\max \{1, d-2\bigr\}}, 
\end{equation}
where ${\bigtriangleup}$  is  the {\it Laplace operator\/}  acting in the sense of the  theory of distribution or generalized functions, and 
$\Gamma$ is the gamma function. If $u\equiv -\infty$ on $C\in \Conn_{\RR^d_{\infty}}(O)$, then  we set $\varDelta_{-\infty}(S):=+\infty$ for each  $S\subset C$. Given $S\subset \RR^d$, we set 
\begin{equation*}
\begin{split}
\text{Sbh}(S)&:=\bigcup \Bigl\{\sbh(O')\colon 
S\subset O'\overset{\text{\tiny open}}{=}\Int O'\subset \RR^d\Bigr\},
\\
\text{Sbh}_*(S)&:=\bigcup \Bigl\{\sbh_*(O')\colon 
S\subset O'\overset{\text{\tiny open}}{=}\Int O'\subset \RR^d\Bigr\},
\\
\text{Har}_*(S)&:=\bigcup \Bigl\{\har(O')\colon 
S\subset O'\overset{\text{\tiny open}}{=}\Int O'\subset \RR^d\Bigr\}.
\end{split}
\end{equation*}
Consider a  binary relation $\cong\,\subset  \text{Sbh}(S)\times \text{Sbh}(S)$
on $\text{Sbh}(S)$ defined by the rule: $U\cong V$ if {\it there is an open set $O'\supset S$ in\/ $\RR^d$ such that $U\in \sbh (O')$, $V\in \sbh(O')$, and $U(x)=V(x)$ for each $x\in O'$.}  This relation  $\cong$ is an {\it equivalence relation\/} on  $\text{Sbh}(S)$, on $\text{Sbh}_*(S)$, and on $\text{Har}(S)$.
The quotient sets of $\text{Sbh}(S)$, of $\text{Sbh}_*(S)$, and of $\text{Har}(S)$ 
by $\cong$  are denoted below by 
$\sbh(S):=\text{Sbh}(S)/\cong$,  $\sbh_*(S):=\text{Sbh}_*(S)/\cong$,
and $\har(S):=\text{Har}(S)/\cong$,  respectively.  The equivalence class $[u]$ of $u$ is denoted without square brackets as simply $u$, and we do not distinguish between the equivalence class $[u]$ and the function $u$ when possible. 
So, for $u,v\in \sbh(S)$, we write {\it``$u=v$ on $S$''} if $[u]=[v]$ in $\sbh(S):=\text{Sbh}(S)/\cong$, or, equivalently, $u\cong v$ on $\text{Sbh}(S)$,
and  we write $u\not\equiv-\infty$ if $u\in \sbh_*(S)$.
The concept of the Riesz measure $\varDelta_u$ of $u\in \sbh(S)$ is correctly and uniquely defined by the restriction $\varDelta_u\bigm|_S$ of the Riesz measure $\varDelta_u$ to $S$.  For $u\in \sbh(S)$ and $v\in \sbh(S)$, the concepts {\it ``$u\leq v$ on $S$'',\/ {\rm and}  ``$u=v$ outside $S$'', ``$u\leq v$ outside $S$'', ``$u$ is harmonic outside S'',\/} etc. defined naturally:
$u(x)\leq v(x)$ {\it for each} $x\in S$, and  {\it there exits an open set 
$O'\supset S$} {\rm such that } $u(x)=v(x)$ {\it for each $x\in O'\setminus S$}, $u(x)\leq v(x)$ {\it for each $x\in O'\setminus S$}, {\it the restriction $u\bigm|_{O'\setminus S}$ is harmonic on 
$O'\setminus S$,} respectively.

\subsection{Potentials}\label{Pjhar}
For a {\it charge\/} $\mu\in \Meas_{\comp}(O)$ its  {\it potential\/}  
\begin{equation}\label{pot} 
\pt_{\mu}\colon \RR^d_{\infty}\to \overline \RR, \quad \pt_{\mu}(y)
\overset{\eqref{kK}}{:=}\int_O K_{d-2}(x,y) \dd \mu (x), 
\end{equation}
is uniquely determined on \cite{Arsove}, \cite[3.1]{KhaRoz18}
\begin{equation}\label{Dom}
\Dom \pt_{\mu} 
:=\left\{y\in \RR^d\colon
\inf\left\{ \int_{0}^1\frac{\mu^-(y,t)}{t^{d-1}} \dd t, \int_{0}^1\frac{\mu^+(y,t)}{t^{d-1}} \dd t\right\}<+ \infty 
\right\}
\end{equation}
by values in  $\overline \RR$, 
and  the set $E:=(\complement \Dom \pt_{\mu})\setminus \infty$ is {\it polar\/} with zero {\it outer capacity\/} 
\begin{equation*}
\text{Cap}^*(E):=\inf_{E\subset O'\overset{\text{\rm \tiny open}}{=}\Int O'}  
\sup_{\stackrel{C\overset{\text{\tiny closed}}{=}\clos C\overset{\text{\tiny compact}}{\Subset} O}{\nu\in \Meas^{1+}(C)}} 
 k_{d-2}^{-1}\left(\iint K_{d-2} (x,y)\dd \nu (x) \dd \nu(y) \right).
\end{equation*}
Evidently $\pt_{\mu}\in \har\bigl(\RR^d\setminus \supp |\mu|\bigr)$, and if $\mu \in \Meas^+_{\comp}(\RR^d)$, then $\pt_{\mu}\in \sbh_*(\RR^d)$.

\subsection{ In detail on the classical Poisson\,--\,Jensen formula}\label{ccPJ}

If  $x\in D\Subset O$, then  
the {\it extended harmonic measure  $\omega_D(x, \cdot )\in \Meas^{1+}(\partial D)\subset \Meas^{1+}_{\comp}(\RR^d)$ (for $D$ at\/} $x$)  defined  on sets $B\in \Borel(\RR^d)$ by
\begin{equation}\label{oB}
\omega_D(x, B):=\sup\left\{u(x)\colon u\in \sbh(D),\; \limsup_{D\ni y'\to y\in \partial D } u(y')\leq 
\begin{cases}
1\text{ for $y\in B\cap \partial D$}\\
0\text{ for $y\notin B\cap \partial D$}
\end{cases} \hspace{-2mm}
\right\}
\end{equation}

the potential (see \cite[Ch.~4,\S~1,2]{Landkoff}) 
\begin{multline}\label{Dog}
\pt_{\omega_D(x, \cdot )-\delta_x}(y) =
\pt_{\omega_D(x, \cdot)}(y)-\pt_{\delta_x}(y) 
\\=\int_{\partial D} K_{d-2}(y,x') \dd_{x'} \omega_D(x, x')-K_{d-2}(y,x )
=g_D(y,x), \quad y\in \RR^d_{\infty}, \quad x\in D,  
\end{multline}
is equal to  the {\it generalized Green's function $g_D(\cdot,x)\colon \RR^d_{\infty}\to \overline \RR^+$ (for $D$ with pole at $x$ and  $g_D(x,x):=+\infty$)\/}
defined on $\RR^d_{\infty}\setminus x$ by upper semicontinuous regularization
 \begin{equation}\label{gD}
\begin{split}
g_D(y ,x)&:=\check{g}^*(y,x)
:=\limsup_{\RR^d\ni y'\to y} \check{g}(y',x)\in \overline{\RR}^+ \quad 
\text{for each $y\in \RR^d_{\infty}\setminus x$, where} \\
\check{g}(y,x)&:=\sup\left\{u(y)\colon
u\in \sbh(\RR^d\setminus x), \;
\begin{cases}
u(y')\leq 0 \text{ for each  }y\notin \clos D, \\
\limsup\limits_{x\neq y\to x}\dfrac{u(y)}{-K_{d-2}(x,y)}\leq 1
\end{cases} 
\right\}. 
\end{split}
\end{equation}


\section{
Representations for pairs of subharmonic functions}\label{RSF}

\begin{proposition}\label{prO} If\/ $\mu \in \Meas_{\comp}(\RR^d)$, then
\begin{subequations}\label{pnu0}
\begin{align}
{\pt}_{\mu}&\in \sbh(\RR^d)\bigcap \har(\RR^d\setminus \supp \mu), 
\tag{\ref{pnu0}h}\label{{pnu0}h}
\\ 
{\pt}_{\mu}(x)&\overset{ \eqref{{kK}k}}{=}\mu (\RR^d)k_{d-2}\bigl(|x|\bigr)+O\bigl(1/|x|^{d-1}\bigr),  
\quad x\to \infty.
\tag{\ref{pnu0}$\infty$}\label{{pnu0}infty}
\end{align}
\end{subequations}
\end{proposition}

\begin{proof} For $d=1$, we have
\begin{equation*}
\left|\pt_{\mu}(x)-\mu(\RR)|x|\right|\leq \int \bigl||x-y|-|x|\bigr|\dd |\mu|(y)\leq \int |y|\dd |\mu|(y)
=O(1), \quad |x|\to +\infty. 
\end{equation*} 
 
See \cite[Theorem 3.1.2]{R} for $d=2$. 

For $d>2$ and $|x|\geq 2\sup\bigl\{|y|\colon y\in \supp \mu\bigr\}$, we have 
\begin{multline*}
\left|{\pt}_{\mu}(x)-\mu (\RR^d)k_{d-2}\bigl(|x|\bigr)\right|
=\left|\int \left(\frac{1}{|x|^{d-2}}-\frac{1}{|x-y|^{d-2}}\right)
\dd \mu (y)\right|\\
\leq\int \left|\frac{1}{|x|^{d-2}}-\frac{1}{|x-y|^{d-2}}\right|
\dd |\mu|(y)
\leq \frac{2^{d-2}}{|x|^{2d-4}}\int \left||x-y|^{d-2}
-|x|^{d-2}\right| \dd |\mu|(y)
\\
\leq \frac{2^{d-2}}{|x|^{2d-4}}
\int |y||x|^{d-3} \sum_{k=0}^{d-3}\Bigl(\frac32\Bigr)^k 
\dd |\mu|(y)\leq 2\frac{3^{d-2}}{|x|^{d-1}}
\int |y|\dd |\mu|(y)=O\Bigl(\frac{1}{|x|^{d-1}}\Bigr),
\end{multline*}
and we get \eqref{{pnu0}infty}.
\end{proof}

\begin{theorem}\label{lemPQ} Let $O\subset \RR^d$ be an open set, and let  $p\in \sbh_*(O)$ and $q\in \sbh_*(O)$ be pair of functions such that $p$ and $q$ are harmonic outside a compact subset in $O$. If there is a compact set $S\Subset O$ such that $p=q$ on $O\setminus S$, then, for Riesz measures $\varDelta_p\in \Meas_{\comp}^+(O)$ of $p$ and  $\varDelta_q\in \Meas_{\comp}^+(O)$ of $q$, we have 
\begin{equation}\label{ptd}
\varDelta_p(O)=\varDelta_q(O), 
\quad \text{$\pt_{\varDelta_p}=\pt_{\varDelta_q}$ on $\RR^d\setminus S$},
\end{equation}
and there is a harmonic function $H$ on $O$  such that 
\begin{equation}\label{PQpR}
\begin{cases}
p=\pt_{\varDelta_p}+H\\
q=\pt_{\varDelta_q}+H
\end{cases}
\quad \text{on $O$, \quad  $H\in \har(O)$}.
\end{equation}
\end{theorem}
\begin{proof} By Weyl's lemma on the Laplace equation, we have  
\begin{equation*}
\begin{cases}
\bigtriangleup(p-\pt_{\varDelta_p})\overset{\eqref{df:cm}}{=}\frac{1}{c_d}(\varDelta_p-\varDelta_p)=0\\
\bigtriangleup(q-\pt_{\varDelta_q})\overset{\eqref{df:cm}}{=}\frac{1}{c_d}(\varDelta_q-\varDelta_q)=0
\end{cases}\quad \Longrightarrow \quad 
\begin{cases}
h_p:=p-\pt_{\varDelta_p}\in \har(O)\\
h_q:=q-\pt_{\varDelta_q}\in \har(O)
\end{cases}
\end{equation*}
 and  obtain representations
\begin{equation}\label{hPQ}
\begin{cases}
p=\pt_{\varDelta_p}+h_p\\
q=\pt_{\varDelta_Q}+h_q
\end{cases}
\text{on $O$  with $h_p\in \har(O)$ and $h_q\in \har(O)$}.
\end{equation}

Let us first consider separately 

\paragraph{\bf The case $O:=\RR^d$ in the notation $P: = p$ and $Q: = q$.} 
Put 
\begin{equation}\label{hpq}
h\overset{\eqref{hPQ}}{:=}h_P-h_Q\in \har(\RR^d).
\end{equation}
By the conditions of Theorem \ref{lemPQ} and Proposition \ref{prO}, we have
\begin{multline}\label{bd}
h(x)\overset{\eqref{hpq}}{=}h_P(x)-h_Q(x)\overset{\eqref{hPQ}}{=}-\pt_{\varDelta_P}(x)+\pt_{\varDelta_Q}(x)+\bigl(P(x)-Q(x)\bigr)\\
\overset{\eqref{{pnu0}infty}}{=}
bk_{d-2}\bigl(|x|\bigr)
+O\bigl(|x|^{1-d}\bigr), \quad |x|\to +\infty, \quad\text{where  $b:=\varDelta_Q(\RR^d)-\varDelta_P(\RR^d)$.}
\end{multline}

\paragraph{The case $d>2$.} If $d\geq 3$, then, in view of \eqref{bd}, this harmonic function $h$ bounded 
on $\RR^d$. By Liouville's Theorem \cite[Ch.~3]{ABR}, $h$ is constant, and 
$h_P-h_Q=h\overset{\eqref{bd}}{\equiv} 0$ on $\RR^d$. 
In particular, $|b|=\bigl|b+|x|^{d-2}h(x)\bigr|\overset{\eqref{bd}}{=}
O\bigl(1/|x|\bigr) $ as $x\to \infty$, i.e., $b=0$. 
Thus, for   $H:=h_P=h_Q$,   
 by \eqref{hPQ}, we obtain representations \eqref{PQpR} together with  $\pt_{\varDelta_P}=\pt_{\varDelta_Q}$ on $\RR^d\setminus S$,  as required. 
  
\paragraph{The case $d=2$.} Using  \eqref{bd} we obtain
$\bigl|h(x)-b\log |x|\bigr|\overset{\eqref{bd}}{=}O\bigl(1/|x|\bigr)$ as $x\to \infty$. Hence, this harmonic function $h$ is bounded from below
 if $b\geq 0$ or  bounded from above if $b<0$. Therefore, by  Liouville's Theorem,  $h$ is constant, $b=0$, i.e., $\varDelta_P(\RR^2)\overset{\eqref{bd}}{=}\varDelta_Q(\RR^2)$, and $h
\overset{\eqref{bd}}{\equiv} 0$ on $\RR^2$. Thus, we obtain  \eqref{PQpR} together with \eqref{ptd}.  
 \paragraph{The case $d=1$.} Using  \eqref{bd} we obtain
$\bigl|h(x)-b|x|\bigr|\overset{\eqref{bd}}{=}O(1)$ as $x\to \infty$. Hence, this affine  function $h$
on $\RR$ is bounded from below  if $b\geq 0$ or bounded from above if $b<0$. 
Therefore, $h$ is constant, $b=0$, i.e., $\varDelta_P(\RR)\overset{\eqref{bd}}{=}\varDelta_Q(\RR)$, and 
$h\overset{\eqref{bd}}{\equiv} C$  on $\RR$ for a constant $C\in \RR$.  Thus, 
 \begin{equation}\label{hPQc}
\begin{cases}
P(x)=\pt_{\varDelta_P}(x)+ax+b+C\\
Q(x)=\pt_{\varDelta_Q}(x)+ax+b
\end{cases}
\text{for  $x\in \RR$  with $h_Q(x)\underset{\text{\tiny $x\in \RR$}}{\equiv} ax+b$},
\end{equation}  
The definition \eqref{pot}  of potentials in the case $d=1$  immediately implies 
\begin{lemma}\label{lem2} Let $\varDelta\in \Meas^+_{\comp}(\RR)$, and
$s_l:=\inf \supp \varDelta$, $s_r:=\sup \supp \varDelta$. 
Then 
\begin{equation*}
\pt_\varDelta(x)=
\begin{cases}
\varDelta(\RR)x-\int y\dd \varDelta(y)&\text{if $x\geq s_r$},\\
-\varDelta(\RR)x+\int y\dd \varDelta(y)&\text{if $x\leq s_l$}.
\end{cases}
\end{equation*}
\end{lemma}
We set \begin{equation*}
\begin{cases}
t:=\varDelta_P(\RR)=\varDelta_Q(\RR)\in \RR^+,\\ 
S_l:=\inf (S\cup \supp \varDelta_P\cup \supp \varDelta_Q)\in \RR,\\
S_r:=\sup (S\cup \supp \varDelta_P\cup \supp \varDelta_Q)\geq S_l.
\end{cases}
\end{equation*}
In view of $P(x)\equiv Q(x)$ for $x\in \RR\setminus S$,   by Lemma  \ref{lem2}, we have    
\begin{equation*}
\begin{cases}
tx-\int y\dd \varDelta_P(y)+ax+b+C=
tx-\int y\dd \varDelta_Q(y)+ax+b \quad\text{if $x\geq S_r$},\\
-tx+\int y\dd \varDelta_P(y)+ax+b+C=
-tx+\int y\dd \varDelta_Q(y)+ax+b \quad\text{if $x\leq S_l$},
\end{cases}
\end{equation*}
whence
\begin{equation*}
\begin{cases}
-\int y\dd \varDelta_P(y)+C=
-\int y\dd \varDelta_Q(y),\\
\int y\dd \varDelta_P(y)+C=
\int y\dd \varDelta_Q(y).
\end{cases}
\end{equation*}
Adding these equalities, we obtain $C=0$. Thus, we get 
 \eqref{PQpR} together with \eqref{ptd}.

\paragraph{\bf The general case of an open set  $O\subset \RR^d$.}
Let's start again with the representations \eqref{hPQ}. We set  
 \begin{subequations}\label{d}
\begin{align}
{\mathsf S}&\overset{\text{\tiny closed}}{:=}S\bigcup \supp \varDelta_q\bigcup \supp \varDelta_p\overset{\text{\tiny compact}}{\Subset} O,
\tag{\ref{d}S}\label{{d}S}\\
w:=p-q,&\quad  \varDelta_w\overset{\eqref{df:cm}}{:=}c_d\bigtriangleup\!w=\varDelta_p-\varDelta_q\in \Meas({\mathsf S})\subset
\Meas_{\comp}(O).
\tag{\ref{d}w}\label{{d}w}
\end{align}
\end{subequations}
This difference $w\in \sbh_*(O)-\sbh_*(O)$ of subharmonic functions, i.e., a $\delta$-subharmonic function   \cite{Arsove}, \cite{Arsove53p}, \cite[3.1]{KhaRoz18},  is uniquely defined on $O$ outside a polar set (cf. \eqref{Dom})
\begin{equation}\label{Domd}
\Dom w 
:=\left\{x\in O\colon
\inf\left\{ \int_{0}\frac{\varDelta_w^-(x,t)}{t^{d-1}} \dd t, \int_{0}\frac{\varDelta_w^+(x,t)}{t^{d-1}} \dd t\right\}<+ \infty 
\right\}\overset{\eqref{{d}S}}{\subset} {\mathsf S},
\end{equation}
and $w\equiv 0$ on $O\setminus {\mathsf S}$ since $p=q$ outside $S\subset {\mathsf S}$ in \eqref{{d}w}, and $p,q\in \har(O\setminus {\mathsf S})$.  The Riesz charge $\varDelta_w
\overset{\eqref{d}}{\in} \Meas_{\comp}(O)$ of this $\delta$-subharmonic function $w$ on $O$ is also uniquely determined on $O$ with $\supp |\varDelta_w|\subset {\mathsf S}$ \cite[Theorem 2]{Arsove}. 
The function $w\colon O\setminus \Dom w\to \overline \RR$ can be extended from $O$ to the whole of $\RR^d\setminus \Dom w$ 
 by zero values:
\begin{equation}\label{d0}
w\equiv 0 \quad\text{on }\RR^d\setminus {\mathsf S}\overset{\eqref{{d}S}}{\supset} \RR^d\setminus O, 
\quad \varDelta_w=\varDelta_p-\varDelta_q\overset{\eqref{{d}w}}{\in} \Meas({\mathsf S}). 
\end{equation}
This function $w$ on $\RR^d\setminus \Dom w$
 is still a $\delta$-subharmonic function, but already on $\RR^d$, since $\delta$-subharmonic functions are defined locally \cite[Theorem 3]{Arsove}. The 
Riesz charge of this  $\delta$-subharmonic function $w\colon \RR^d\setminus \Dom d \to \overline \RR$ on $\RR^d$ is the same charge
$\varDelta_d\overset{\eqref{{d}w}}{\in} \Meas({\mathsf S})$. There is a canonical representation  \cite[Definition 5]{Arsove} of $w$ such that \cite[Theorem 5]{Arsove}
\begin{subequations}\label{PQ}
\begin{align}
w=P-Q &\quad\text{on }\RR^d\setminus \Dom w, \quad\text{where }
P,Q\in \sbh_*(\RR^d)\cap\har(\RR^d\setminus {\mathsf S})
\tag{\ref{PQ}d}\label{{PQ}s}
\\
\intertext{are functions with Riesz measures} 
&\begin{cases}
\varDelta_P\overset{\eqref{df:cm}}{:=}c_d\bigtriangleup\!P=\varDelta_w^+
\overset{\eqref{{PQ}s}}{\in} \Meas^+({\sf S}),\\ 
\varDelta_Q\overset{\eqref{df:cm}}{:=}c_d\bigtriangleup\!Q=\varDelta_w^-\overset{\eqref{{PQ}s}}{\in} \Meas^+({\sf S}),
\end{cases}
\tag{\ref{PQ}$\varDelta$}\label{{PQ}D}
\\
&P\overset{\eqref{d0},\eqref{{PQ}s}}{\equiv}Q \quad \text{on }\RR^d\setminus {\mathsf S},
\tag{\ref{PQ}$\equiv$}\label{{PQ}d}
\\
\intertext{and there is a function $s\in \sbh_*(O)$ with Riesz measure}
&\varDelta_s=\varDelta_p-\varDelta_w^+\overset{\eqref{d0},\eqref{{PQ}D}}{=}\varDelta_q-\varDelta_w^- \in \Meas^+({\mathsf S})
\tag{\ref{PQ}s}\label{{PQ}vD}
\\
&\quad\text{such that } \begin{cases}
p=P+s,\\
q=Q+s
\end{cases}
\quad \text{on } O.
\tag{\ref{PQ}r}\label{{PQ}r}
\end{align}
\end{subequations}
By \eqref{{PQ}s} and \eqref{{PQ}d},  all conditions  of Theorem \ref{lemPQ} 
are fulfilled for functions $P,Q$ from \eqref{PQ} instead of $p,q$, but in the case  $\RR^d$ instead of $O$ and ${\mathsf S}$ instead of $S$.  Thus,  we have 
\eqref{ptd} in the form 
\begin{subequations}\label{dpt}
\begin{align}
\varDelta_w^+(O)\overset{\eqref{{PQ}D}}{=}\varDelta_P(\RR^d)
\overset{\eqref{ptd}}{=}\varDelta_Q(\RR^d)
\overset{\eqref{{PQ}D}}{=}\varDelta_w^-(O), 
\tag{\ref{dpt}$\varDelta$}\label{{dpt}d}
\\
\pt_{\varDelta_w^+}\overset{\eqref{{PQ}D}}{=}\pt_{\varDelta_P}=\pt_{\varDelta_Q}\overset{\eqref{{PQ}D}}{=}\pt_{\varDelta_w^-} \quad \text{on $\RR^d\setminus {\mathsf S}$},
\tag{\ref{dpt}p}\label{{dpt}pt}
\end{align}
\end{subequations}
and  the representations \eqref{PQpR} in the form
\begin{equation}\label{PQpR+}
\begin{cases}
P\overset{\eqref{PQpR}}{=}\pt_{\varDelta_P}+h
\overset{\eqref{{dpt}pt}}{=}\pt_{\varDelta_w^+}+h\\
Q\overset{\eqref{PQpR}}{=}\pt_{\varDelta_Q}+h
\overset{\eqref{{dpt}pt}}{=}\pt_{\varDelta_w^-}+h
\end{cases}
\quad \text{on $\RR^d$, \quad  $h\in \har(\RR^d)$}.
\end{equation}
Hence, by representation  \eqref{{PQ}r},  we obtain the following representations
\begin{equation}\label{rPQ}
\begin{split}
&\begin{cases}
p\overset{\eqref{{PQ}r},\eqref{PQpR+}}{=}\pt_{\varDelta_w^+}+h+s,\\
q\overset{\eqref{{PQ}r},\eqref{PQpR+}}{=}\pt_{\varDelta_w^-}+h+s
\end{cases}
\quad \text{on  $O$}, \\
h\in \har(\RR^d),\quad &\pt_{\varDelta_w^+}\overset{\eqref{{dpt}pt}}{=}\pt_{\varDelta_w^-}
\text{ on $\RR^d\setminus {\mathsf S}$,}\quad
\varDelta_w^+(O)\overset{\eqref{{dpt}d}}{=}\varDelta_w^-(O). 
 \end{split}
\end{equation} 
Besides,  the function  $l\overset{\eqref{{PQ}vD}}{:=}s-\pt_{\varDelta_s}$ 
is harmonic on $O$ by Weyl's lemma on the Laplace equation 
$\bigtriangleup\,(s-\pt_{\varDelta_s})\overset{\eqref{{PQ}vD}}{=}\varDelta_s-\varDelta_s=0$.  Hence 
\begin{equation}\label{rPQs}
\begin{split}
&\begin{cases}
p\overset{\eqref{rPQ}}{=}\pt_{\varDelta_w^+}+\pt_{\varDelta_s}+h+l,\\
q\overset{\eqref{rPQ}}{=}\pt_{\varDelta_w^-}+\pt_{\varDelta_s}+h+l
\end{cases}
\quad \text{on  $O$, where $h\in \har(\RR^d)$ and  $l\in \har(O)$,} 
\\
& \pt_{\varDelta_w^+}+\pt_{\varDelta_s}\overset{\eqref{rPQ}}{=}\pt_{\varDelta_w^-}+\pt_{\varDelta_s}\text{ on } \RR^d\setminus {\mathsf S},\quad 
\varDelta_w^+(O)\overset{\eqref{rPQ}}{=}\varDelta_w^-(O). 
\end{split}
\end{equation}
By construction, we have
\begin{equation*}
\begin{split}
&\begin{cases}
\pt_{\varDelta_w^+}+\pt_{\varDelta_s}=\pt_{\varDelta_w^++\varDelta_s}
\overset{\eqref{{PQ}vD}}{=}\pt_{\varDelta_p},
\\
\pt_{\varDelta_w^-}+\pt_{\varDelta_s}=\pt_{\varDelta_w^-+\varDelta_s}
\overset{\eqref{{PQ}vD}}{=}\pt_{\varDelta_q},
\end{cases} \\
&\varDelta_p(O)=(\varDelta_w^++\varDelta_s)(O)
\overset{\eqref{{PQ}vD}}{=}
(\varDelta_w^-+\varDelta_s)(O)=\varDelta_p(O).
\end{split}
\end{equation*}
Hence, if we set $H:=h+l\in \har(O)$, then, by \eqref{rPQs}, we obtain 
exactly \eqref{PQpR}, as well as \eqref{ptd}, with the only difference being that in \eqref{ptd} we have ${\mathsf S}\overset{\eqref{{d}S}}{\supset} S$ instead of $S$. 
Moreover, it immediately follows from the representation \eqref{PQpR} and the condition $p = q$ on ${\mathsf S}\setminus S\overset{\eqref{{d}S}}{\subset} O\setminus S$ that $\pt_{\varDelta_p}=\pt_{\varDelta_q}$ on $\RR^d\setminus S=
(\RR^d\setminus {\mathsf S})\bigcup ({\mathsf S}\setminus S)$.
\end{proof}





\def\bibname{\vspace*{-30mm}{

\vskip 1 cm \footnotesize
\begin{flushleft}
Bulat~N.~Khabibullin  \\ 
Institute of Mathematics with Computing Centre - Subdivision of the Ufa Federal Research Centre of the Russian Academy of Sciences\\ 
112, Chernyshevsky str., Ufa, Russia, 450008 \\ 
E-mail: Khabib-Bulat@mail.ru 
\end{flushleft}


\end{document}